\documentclass[10pt]{amsart}
\usepackage{amssymb,amsbsy,amsmath,amsfonts,times, amscd}
\usepackage{latexsym,euscript,exscale}

\usepackage{helvet}

\title{ Lipschitz structure and minimal metrics on topological groups}

\author {Christian Rosendal}
\address{Department of Mathematics, Statistics, and Computer Science (M/C 249)\\University of Illinois at Chicago\\851 S. Morgan St.\\Chicago, IL 60607-7045\\USA}
\email{rosendal.math@gmail.com}
\urladdr{http://homepages.math.uic.edu/$~$rosendal}

\date {}
\newcommand{\norm}[1]{\lVert#1\rVert}

\newcommand {\N}{\mathbb N}

\newcommand {\R}{\mathbb R}
\newcommand {\Z}{\mathbb Z}

\newcommand {\T}{\mathbb T}

\newcommand{\eps}{\epsilon}

\newcommand{\saa}{\Rightarrow}

\newcommand{\til}{\rightarrow}

\newcommand{\Lim}[1]{\mathop{\longrightarrow}\limits_{#1}}

\newcommand {\del}{ \; \big| \;}

\newcommand {\ku} {\mathcal}

\newcommand{\ov}{\overline}
\newcommand{\inv}{^{-1}}
\newcommand{\Id}{{\rm Id}}

\newcommand {\e} {\exists}
\renewcommand {\a} {\forall}

\newtheorem{thm}{Theorem}
\newtheorem{cor}[thm]{Corollary}
\newtheorem{lemme}[thm]{Lemma}
\newtheorem{prop} [thm] {Proposition}
\newtheorem{defi} [thm] {Definition}

\newtheorem{prob}[thm]{Problem}

\theoremstyle{definition}

\newtheorem{exa}[thm]{Example}

\usepackage[sc]{mathpazo}

\begin{document}

\thanks{The author was partially supported by a Simons Foundation Fellowship (Grant \#229959) and also recognises support from the NSF (DMS 1201295 \& DMS  1464974)}

\begin{abstract}
We discuss the problem of deciding when a metrisable topological group $G$ has a canonically defined local Lipschitz geometry. This naturally leads to the concept of {\em minimal} metrics on $G$, that we characterise intrinsically in terms of a linear growth condition on powers of group elements. 

Combining this with work on the large scale geometry of topological groups, we also identify the class of metrisable groups admitting a canonical global Lip\-schitz geometry.

In turn, minimal metrics connect with Hilbert's fifth problem for completely metrisable groups and we show, assuming that the set of squares is sufficiently rich, that every element of some identity neighbourhood belongs to a $1$-parameter subgroup. 
\end{abstract}

\maketitle


The present note deals with the problem of deciding which metrisable topological groups have a well-defined local geometry intrinsic to the topological group structure. To make this problem more precise, let us recall that a {\em metrisable topological group} is a topological group $G$ whose topology may be induced by some metric, which then is said to be {\em compatible} with the topology on $G$. Thus, the metric itself is not part of the given data. These groups where characterised in fundamental papers by G. Birkhoff \cite{birkhoff} and S. Kakutani \cite{kakutani}, namely, a Hausdorff topological group $G$ is metrisable if and only if it is first countable. Moreover, such a group necessarily admits a compatible {\em left-invariant} metric $d$, i.e., so that $d(hg,hf)=d(g,f)$ for all $g,f,h\in G$.

An easy calculation shows that, if $d$ and $\partial$ are compatible left-invariant metrics on a topological group $G$, then the identity map ${\rm id}\colon (G,\partial)\til (G,d)$ is always uniformly continuous and hence, by symmetry, a uniform homeomorphism. This is of course also a reflection of the fact that both $d$ and $\partial$ will metrise the left-uniform structure on $G$. However, unless further assumptions are added,  there is in general no control on the modulus of uniform continuity of the mapping. The problem is thus to decide which, if any, of the compatible left-invariant metrics on $G$ determine a canonical local geometric structure on the group. At least up to local bi-Lipschitz equivalence, this is solved if $G$ admits a {\em minimal} metric in the following sense.

\begin{defi}\label{defi min}
A  metric $d$ on a topological group $G$ is said to be 
{\em minimal} if it compatible, left-invariant and, for every other compatible left-invariant metric $\partial$ on $G$, the map
$$
{\rm id}\colon (G,\partial)\til (G,d)
$$
is Lipschitz in a neighbourhood of the identity, i.e., if there is an identity neighbourhood $U$ and a constant $K$ so that
$$
d(g,f)\leqslant K\cdot \partial(g,f)
$$
for all $g,f\in U$.
\end{defi}
Let us first observe that, if $U$ and $K$ are as above, then ${\rm id}\colon (G,\partial)\til (G,d)$ is locally $K$-Lipschitz. For given $h\in G$ and $v,w\in U$, note that
$$
d(hv,hw)=d(v,w)\leqslant K\cdot \partial(v,w)=\partial(hv,hw),
$$
so the identity map is $K$-Lipschitz on the neighbourhood $hU$ of $h$.
It follows immediately  that any two minimal metrics on $G$ are locally bi-Lipschitz and  thus identify a canonical local geometric or, more specifically,  Lipschitz structure on $G$.

The concept of an intrisic Lipschitz structure on a topological object is of course common to other areas. For example, a well-known result due to D. Sullivan \cite{sullivan} states that, except for $n=4$, any topological $n$-manifold $M$ admits a Lipschitz structure, that is, an atlas $\{\phi_i\colon U_i\to \R^n\}$ whose transition maps are locally Lipschitz. Moreover, any two such Lipschitz structures are related by a locally bi-Lipschitz homeomorphism of $M$.  The local Lipschitz structure identified by a minimal metric is even more rigid, since any two minimal metrics are locally bi-Lipschitz by the identity map.

We remark that, unless we accept to force the metric $d$ to be bounded, the local minimality of Definition \ref{defi min} really describes the strongest notion of minimality possible. Indeed, if $d$ is unbounded, then $\sqrt d$ is a compatible left-invariant metric, while  ${\rm id}\colon (G,\sqrt d)\til (G,d)$ is not Lipschitz for large distances.

Note also that, at least for short distances, there is no maximal metric unless $G$ is discrete. That is, if $G$ is non-discrete and $d$ is any compatible left-invariant metric, then the mapping ${\rm id}\colon (G,d)\til (G,\sqrt d)$ will not be Lipschitz for short distances.

Another helpful observation is that, if $d$ is minimal and $U\ni1 $ is a fixed $d$-bounded open neighbourhood, then, for every compatible left-invariant $\partial$, the map
$$
{\rm id}\colon (U,\partial)\til (U,d)
$$
is Lipschitz, i.e., the localising set $U$ is independent of $\partial$. It follows immediately that any two bounded minimal metrics are bi-Lipschitz equivalent. Nevertheless, in many cases, a better global Lipschitz structure can be identified that captures the large scale geometry of the group. We return to this in Theorem \ref{min-max}.

One way to think about minimal metrics is via the growth or rather decay of the balls $ B_d(\alpha)=\{g\in G\del d(g,1)\leqslant \alpha\}$ as $\alpha\til0$. Namely, minimality of $d$ simply expresses that, if  $\partial$ is another compatible left-invariant metric on $G$, then, for some $K=K(\partial)\geqslant 1$, we have
$$
B_\partial(\alpha)\subseteq B_d({K\alpha})
$$
whenever $\alpha\leqslant 1$. So, up to the rescaling by $K$, the $d$-balls $B_d(\beta)$ are as large as possible as $\beta$ decreases to $0$.

Whereas minimality of a metric is a relative notion, i.e., defined in terms of comparisons with other compatible left-invariant metrics on the group, the main result of our note furnishes an internal characterisation of minimality without reference to other metrics. Namely, we characterise the minimal metrics as those satisfying a certain linear growth condition on powers in a neighbourhood of $1$. This condition in turn has already been studied in the literature in the context of locally compact groups, where it turned out to be central to the solution to Hilbert's fifth problem. We shall discuss this connection after our result.

\begin{thm}\label{main}
The following conditions are equivalent for a compatible left-invariant metric $d$ on a topological group $G$.
\begin{enumerate}
\item $d$ is minimal,
\item\label{strongest cond} there is an open set $U\ni1 $ so that
$$
g, g^2,g^3, \ldots, g^n\in  U\;\;\saa\;\; d(g,1)\leqslant \frac 1{n},
$$
\item\label{gleason cond} there are constants $\eps>0$ and $K\geqslant 1$ so that
$$  
d(g,1)\leqslant \frac\eps n\;\;\saa\;\;  n\cdot d(g,1) \leqslant   K\cdot d(g^n,1),
$$
\item there are an open set $U\ni 1$ and a constant $K\geqslant 1$ so that
$$
g,g^2,g^4, g^8,\ldots, g^{2^n}\in U\;\;\saa\;\;  2^n\cdot d(g,1) \leqslant   K\cdot d(g^{2^n},1).
$$
\end{enumerate}
\end{thm}

The above result may be said to provide a satisfying description of minimal metrics on the group, indeed, the criterion only involves computations with powers of single elements. On the other hand, we have no informative reformulation of which metrisable groups admit minimal metrics. One would like to know if there is such a description that does not directly involve asking for an object as complicated as a minimal metric itself. More precisely, the following problem remains open.
\begin{prob}
Let $G$ be a universal Polish group, e.g., $G={\rm Homeo}([0,1]^\N)$. Is the collection
$$
\{H\leqslant G\del H \text{ is a closed subgroup admitting a minimal metric}\}
$$
Borel in the  standard Borel space of closed subgroups of $G$?
\end{prob}

Before commencing the proof of Theorem \ref{main}, we recall some procedures for constructing compatible left-invariant metrics on a topological group. The main result in this area is the above mentioned theorem independently due to Birkhoff and Kakutani. Of the two proofs, Birkhoff's is the simplest and relies on a memorable little  trick.
\begin{lemme}[G. Birkhoff \cite{birkhoff}]\label{birkhoff} 
Let $G$ be a topological group and $\{V_{3^n}\}_{n\in\Z}$
a neighbourhood basis at the identity consisting of symmetric open sets so that $G=\bigcup_{n\in \Z}V_{3^n}$ and $\big(V_{3^n}\big)^3\subseteq V_{3^{n+1}}$.
Define $\delta(g,f)=\inf\big(2^n\del f\inv g\in V_{3^n}\big)$ and put
$$
d(g,f)=\inf \Big(\sum_{i=0}^{k-1}\delta(h_i,h_{i+1})\del h_0=g,
h_k=f\Big).
$$
Then
$\delta(g,f)\leqslant 2\cdot d(g,f)\leqslant 2\cdot\delta(g,f)$
and $d$ is a compatible left-invariant metric on $G$.
\end{lemme}
However, the metric $d$ produced by Birkhoff's construction decreases exponentially faster than needed for our purposes due to a factor $\big(\frac 32)^n$.  For this, we shall instead rely on the construction of Kakutani from which a better estimate can be extracted (see \cite{diestel} for a  proof of the exact statement of Lemma \ref{kakutani} below).

\begin{lemme}[S. Kakutani \cite{kakutani}]\label{kakutani}
Let $G$ be a topological group and $\big\{V_{2^{-n}}\big\}_{n\in \N}$ a neighbourhood basis at the identity consisting of symmetric open sets satisfying $\big(V_{2^{-n}}\big)^2\subseteq V_{2^{-n+1}}$.
Then there is a compatible left-invariant metric $d$ on $G$  so that
$$
B_d({2^{-n}})\subseteq V_{2^{-n}}\subseteq B_d({4\cdot 2^{-n}})
$$
for all $n\in \N$.
\end{lemme}

We now turn to the proof of Theorem \ref{main}.
\begin{proof}
(1)$\saa$(2): We claim first that there is an open  neighbourhood $U\ni 1$ so that, for all $k\geqslant 1$ and $g\in G$, 
$$
 g, g^2, g^3, \ldots, g^{2^k}\in U\;\;\saa\;\;   g\in B_d(2^{-k}).
$$

In order to see this, we assume the contrary.
Let $V_{2^{-0}}=G$ and, for $m\geqslant 1$, inductively define symmetric open sets $V_{2^{-m}}\ni 1$ as follows. 

Assume that $V_{2^{-m}}$ is the last term that has been defined thus far and let $n\geqslant m$ be large enough so that $B_d(n\inv)\subseteq V_{2^{-m}}$. Since the claim fails for $U=B_d(n\inv )$, there are  $k\geqslant 1$ and  $g\notin B_d(2^{-k})$ so that
$$
g, g^2, g^3, \ldots, g^{2^k}\in B_d(n\inv).
$$
Let $F=\{1,g, g\inv\}$ and note that, since $B_d(n\inv)$ is symmetric, $F^{2^k}\subseteq B_d(n\inv )$. Therefore, as $F$ is finite, we can pick a sufficiently small symmetric open $W\ni1$ so that also $(WFW)^{2^k}\subseteq  B_d(n\inv )\subseteq V_{2^{-m}}$. We then set
$$
V_{2^{-m-k}}=WFW,\quad V_{2^{-m-k+1}}=(WFW)^2,\quad \ldots\quad , V_{2^{-m-1}}=(WFW)^{2^{k-1}}.
$$
At the next stage, we begin with the term $V_{2^{-m-k}}$ and proceed as above. 

Therefore, at the end of the construction, we have a sequence $G=V_{2^{-0}}\supseteq V_{2^{-1}}\supseteq \ldots$ of symmetric open sets forming a neighbourhood basis at $1$ so that $\big(V_{2^{-m}}\big)^2\subseteq V_{2^{-m+1}}$ for all $m\geqslant 0$.  We now apply Lemma \ref{kakutani} to the sequence $(V_{2^{-m}})_{m\geqslant 0}$ to obtain a compatible left-invariant metric $\partial$ satisfying
$$
B_\partial({2^{-m}})\subseteq V_{2^{-m}}\subseteq B_\partial(4\cdot2^{-m}).
$$

Note now that there are infinitely many $m$ so that some stage in the construction began with the  term $V_{2^{-m}}$. So fix such an $m$ and let $k$ and $g$ be as in the construction step. Then $g\in WFW= V_{2^{-m-k}}\subseteq B_\partial(4\cdot2^{-m-k})$ and $g\notin B_d(2^{-k})$, whence 
$$
2^{m-2}\cdot \partial(g,1)\leqslant  2^{-k}\leqslant  d(g,1).
$$
Therefore, $\partial$ is a compatible left-invariant metric on $G$, but ${\rm id}\colon (G,\partial)\til (G,d)$ is not Lipschitz for short distances, contradicting the minimality of $d$ and thus proving the claim.

So, using the claim,  fix $U\ni1 $ open so that, for all $g\in G$ and $k\geqslant 1$, 
$$
g, g^2, g^3, \ldots, g^{2^k}\in U\;\;\saa \;\;  g\in B_d(2^{-k})
$$
and pick some open $V\ni1$ so that $V^2\subseteq U$. Now suppose $g, g^2, g^3, \ldots, g^m\in V$ for some $m$ and let $k\geqslant 0$ be so that $2^k\leqslant m<2^{k+1}$. Then also $g^{2^k+n}=g^{2^k}g^n\in V^2\subseteq U$ for all $n\leqslant 2^k$ and so $g^i\in U$ for all $i\leqslant 2^{k+1}$. In particular,   $g\in B_d(2^{-k-1})\subseteq B_d(m\inv)$.
In other words, 
$$
g, g^2, g^3, \ldots, g^{m}\in V\;\;\saa \;\;  d(g,1)\leqslant \frac 1m,
$$
which proves (2).


(2)$\saa$(3): Assume that $U\ni1 $ is an in (2). By shrinking $U$ if needed, we may assume that $U=B_d(p\inv)$ for some integer $p\geqslant 1$.
Now, choose a symmetric open $W\ni 1$ so that $W^{2p}\subseteq U$. We claim that 
$$
g, g^2, g^3, \ldots, g^n\in W\;\;\saa\;\; n\cdot d(g^n,1)\leqslant 4p\cdot d(g^n,1),
$$
which thus verifies (3). 

Indeed, suppose $g, g^2, g^3, \ldots, g^n\in W$ and $g\neq 1$. Then also $g, g^2, g^3, \ldots, g^{2pn}\in W^{2p}\subseteq U$ and thus, if $m$ is minimal so that $g^{m+1}\notin U$, we have $m\geqslant 2pn$ and  $d(g, 1)\leqslant m\inv$. Let now $k\geqslant 1$ be such that $kn\leqslant m<m+1\leqslant(k+1)n$. Then 
\[\begin{split}
d\big(g^{(k+1)n},1\big)
&\geqslant d\big(g^{m+1},1\big)-d\big(g^{(k+1)n}, g^{m+1}\big)\\
&= d\big(g^{m+1},1\big)-d\big(g^{(k+1)n-(m+1)},1\big)\\
&\geqslant p\inv-\big[(k+1)n-(m+1)\big]\cdot d(g,1)\\
&\geqslant p\inv-\frac nm,
\end{split}\]
whereby
\[\begin{split}
d(g^n,1)
\geqslant \frac{d\big(g^{(k+1)n},1\big)}{k+1}
\geqslant \frac{\frac1p-\frac nm}{k+1}
\geqslant \frac{\frac1p-\frac n{2pn}}{2k}
= \frac1{4p}\cdot \frac 1k\geqslant \frac1{4p}\cdot \frac nm\geqslant n\frac1{4p}\cdot d(g,1)
\end{split}\]
as claimed.


(3)$\saa$(4): Let $\eps>0$ and $K\geqslant 1$ be as in (3) and set $U=B_d(\frac  {\eps}{2K})$. Now, suppose $g\neq 1$ and that $g,g^2,\ldots ,g^n\in U$ for some $n\geqslant 1$. Let $m\geqslant 1$ be so that $g\in B_d(\frac \eps m)\setminus B_d(\frac \eps{m+1})$. Then, by (3), we have
$$
K\cdot d(g^m,1)\geqslant m\cdot d(g,1)>\eps\frac m{m+1}\geqslant \frac {\eps}2,
$$
i.e., $g^m\notin U$, whereby $m>n$. It thus follows that $d(g,1)\leqslant \frac \eps n$ and hence, by (3) again, that $n\cdot d(g,1)\leqslant K\cdot d(g^n,1)$. In other words, 
$$
g,g^2,\ldots ,g^n\in U\;\;\saa\;\; n\cdot d(g,1)\leqslant K\cdot d(g^n,1).
$$


Now, by shrinking $U$ and increasing $K$, we may suppose that $U=B_d(2^{-k}\big)$  and $K=2^k$ for some $k\geqslant 3$.
Set also $V=B_d(2^{-2k})$, whereby $V^{2^k}\subseteq U$. We claim that, for all $m$ and $g$, 
$$
\big(
\a i\leqslant 2^m\colon g^i\in V
\big)
\;\;\&\; \;
\big(
\a i\leqslant m+2k+1\colon g^{2^i}\in V
\big) 
\;\;\saa\;\; 
\big(
\a i\leqslant 2^{m+1}\colon g^i\in V
\big).
$$ 
Indeed, suppose that $\big(\a i\leqslant 2^m\colon g^i\in V\big)$ and $\big(\a i\leqslant m+2k+1\colon g^{2^i}\in V\big)$. Assume also that $2^m<i<2^{m+2k+1}$ is given. Write $i=2^{p_r}+2^{p_{r-1}}+\ldots+2^{p_1}+j$ for some $m\leqslant p_1<p_2<\ldots<p_r<m+2k+1$ and $j<2^m$ and note that, since $r\leqslant 2k+1\leqslant 2^k-1$, we have
$$
g^i=g^{2^{p_r}}\cdots g^{2^{p_1}}\cdot g^j\in V^rV\subseteq V^{2^k}\subseteq U.
$$
Therefore, as also $g^{2^{m+2k+1}}\in V$, we see that $g^i\in U$ for all $i\leqslant 2^{{m+2k+1}}$.
By the hypothesis on $U$, it follows that 
$$
2^{{m+2k+1}}\cdot d(g,1)\leqslant 2^k\cdot d(g^{2^{m+2k+1}},1)\leqslant 2^k\cdot 2^{-k}=1,
$$ 
i.e., $d(g,1)\leqslant 2^{-({m+2k+1})}$. So, using that $d(g^i,1)\leqslant i\cdot d(g,1)$, this shows that $g^i\in V$ for all $i\leqslant 2^{m+1}$, proving the claim.

Put now $W=B_d(2^{-4k})$ and assume that $n$ and $g$ are given so that 
$$
g, g^2, g^4, g^8, \ldots, g^{2^n}\in W.
$$  
Then, since $W^{2^{2k}}\subseteq V$, we have that $g^{2^i}\in V$ for all $i\leqslant n+2k$. Using our claim to induct on $m=1,2, \ldots, n-1$, we see that $g^i\in V\subseteq U$ for all $i\leqslant 2^n$, and thus $2^n\cdot d(g,1)\leqslant 2^k\cdot d(g^{2^n},1)$, which proves (4).


(4)$\saa$(1): Suppose that $\partial$ is another compatible left-invariant metric on $G$. Fix $\eps,\eta>0$ and $K\geqslant 1$ so that 
$$
g,g^2,g^4, g^8, \ldots, g^{2^n}\in B_d(\eps)\;\;\saa\;\;  2^n\cdot d(g,1) \leqslant  K\cdot d(g^{2^n},1)
$$
and $B_\partial(\eta)\subseteq B_d(\eps)$. Note then that, if $n\geqslant 0$ and $g\in B_\partial({\frac \eta {2^n}})$, we also have  $g,g^2,g^4, g^8, \ldots, g^{2^n}\in B_\partial(\eta)\subseteq B_d(\eps)$, whence 
$$
2^n\cdot d(g,1) \leqslant  K\cdot d(g^{2^n},1)\leqslant K\eps
$$
i.e., $g\in B_d({\frac{K\eps}{2^n}})$. In other words, $B_\partial({\frac \eta {2^n}})\subseteq B_d({\frac{K\eps}{2^n}})$ for all $n\geqslant 0$.

Now, if $g\in B_\partial(\eta)$ is any non-identity element, pick $n\geqslant 0$ so that $g\in B_\partial({\frac \eta {2^n}})\setminus B_\partial({\frac \eta {2^{n+1}}})\subseteq B_d({\frac{K\eps}{2^n}})$. Then, 
$$
d(g,1)
\leqslant{\frac{K\eps}{2^n}} 
=\frac{2K\eps}{\eta}\cdot \frac \eta {2^{n+1}}
\leqslant \frac{2K\eps}{\eta}\cdot \partial(g,1),
$$
showing that the map ${\rm id}\colon (B_\partial(\eta),\partial)\til (B_\partial(\eta),d)$ is $\frac{2K\eps}{\eta}$-Lipschitz.
\end{proof}

Since a compatible left-invariant metric on a topological group $H$ need not extend to a compatible left-invariant metric on a supergroup $G$, it is far from clear from the definition of minimality that the restriction of a minimal metric on $G$ to a subgroup $H$ is also minimal on $H$. However, using instead the reformulations of Theorem \ref{main}, this becomes obvious, whence the following corollary.

\begin{cor}
The class of topological groups admitting minimal metrics is closed under passing to subgroups.
\end{cor}

\begin{exa}
Using the spectral theorem, it is not hard to verify that the metric induced by the operator norm is minimal on the unitary group $\ku U(\ku H)$ of separable infinite-dimensional Hilbert space. Indeed, set
$$
U=\{u\in \ku U(\ku H)\del \norm{u-{\Id}}<1\}
$$
and suppose $u, u^2, u^3, \ldots, u^n\in U$ for some fixed $u\in \ku U(\ku H)$. Then, by the spectral theorem, there is a $\sigma$-finite measure space $(X,\mu)$ and a measurable function $\phi\colon X\til \T$ so that $u$ is unitarily equivalent to the multiplication operator $M_\phi$ on $L^2(X,\mu)$ defined by
$$
\big(M_\phi\xi\big)(x)=\phi(x)\xi(x).
$$
It follows that, for $i=1,\ldots, n$, 
$$
{\rm ess}\sup\{|{\phi^i}(x)-1|\del x\in X\}= \norm{M_{\phi^i-1}}_\infty=\norm{M_\phi^i-\Id}_\infty=\norm{u^i-\Id}<1
$$
and hence that $|\phi(x)-1|<\frac 2n$ for almost all $x\in X$. In other words, $\norm{u-\Id}<\frac 2n$. So
$$
u,u^2, \ldots, u^n\in U\saa  \norm{u-\Id}<\frac 2n,
$$
showing that the metric is minimal.
\end{exa}

Condition (\ref{gleason cond}) of Theorem \ref{main}  has been studied earlier in the literature as part of the solution to Hilbert's fifth problem due to A. Gleason, D. Montgomery, H. Yamabe and L. Zippin. Indeed, in the book \cite{tao} by T. Tao, metrics satisfying this condition are termed {\em weak Gleason} as they underlie A. Gleason's results in \cite{gleason}.  In particular, in \cite{tao} it is shown that a locally compact metrisable group is a Lie group if and only if it has a weak Gleason metric. Moreover, in the locally compact setting, every weak Gleason metric is actually {\em Gleason}, meaning that it satisfies a further estimate on commutators (cf. Theorem 1.5.5 \cite{tao}).

Let us also mention that, if $d$ is a minimal metric on $G$, then there is a constant $K$ and an open set $V\ni 1$ so that
$$
f,g,h\in V\;\;\saa\;\;d(fh,gh)\leqslant K\cdot d(f,g),
$$
i.e., right-multiplication is $K$-Lipschitz in a neighbourhood of $1$.
To see this, let $B_d(\eps)$ be so that $d(g,1)\leqslant 1/n$ whenever $g, g^2, \ldots, g^n\in B_d(\eps)$. Then, given distinct $g,f\in B_d(\frac\eps4)$, let $n\geqslant 1$ be so that $g\inv f\in B_d\big(\frac \eps{n2}\big)\setminus B_d\big(\frac \eps{(n+1)2}\big)$. So, if $h\in B_d(\frac\eps4)$, we have $(h\inv g\inv fh)^i=h\inv (g\inv f)^ih\in B_d(\eps)$ for $i=1,\ldots, n$,
whereby
$$
d(fh,gh)=d(h\inv g\inv fh,1)\leqslant \frac 1n\leqslant \frac 2{n+1}\leqslant \frac 4\eps \cdot d(f,g).
$$

By Condition (\ref{strongest cond}) of Theorem \ref{main}, it is easy to see that, if $G$ is a group with a weak Gleason metric, then $G$ is {\em NSS}, i.e., has {\em no small subgroups}, which simply means that there is a neighbourhood $U\ni 1$ not containing any non-trivial subgroup. Moreover, in the locally compact metrisable case, being NSS is equivalent to being a Lie group and thus also to having a weak Gleason metric (see the exposition in \cite{montgomery} or \cite{tao}). By Theorem \ref{main}, weak Gleason and minimal metrics coincide, but we shall prefer the latter more descriptive terminology.

Whereas a minimal metric establishes a canonical local Lipschitz geometry on a topological group, we will now combine this with the analysis from \cite{coarse} of  the corresponding problem at the large scale.
\begin{defi}
A compatible left-invariant metric $d$ on a topological group $G$ is said to be {\em maximal} if, for every other compatible left-invariant metric $\partial$ on $G$, the map 
$$
{\rm id}\colon (G, d)\to (G,\partial)
$$
is Lipschitz for large distances, that is, $\partial \leqslant K\cdot d+ C$ for some constants $K,C$.
\end{defi}
Clear, any two maximal metrics $d$ and $\partial$ on a topological group $G$ are {\em quasi-isometric}, that is $\frac 1K\cdot d-C\leqslant \partial \leqslant K\cdot d+ C$ for some constants $K,C$. 
\begin{lemme}
Suppose that $d$ and $\partial$ are  both simultaneously minimal and maximal metrics on a topological group $G$. Then $d$ and $\partial$ are bi-Lipschitz equivalent, i.e.,
$$
\frac 1L\cdot d\leqslant \partial \leqslant  L\cdot d
$$
for some constant $L$.
\end{lemme}

\begin{proof}
Since $d$ is minimal, there is an identity neighbourhood $V$ and a constant $K$ so that
$$
{\rm id}\colon (V,\partial)\to (V,d)
$$
is $K$-Lipschitz. On the other hand, as $\partial$ is maximal, there are constants $M,N$ so that $d\leqslant M\cdot \partial +N$. It thus follows that, for $x\in G$ and $v\in V$, 
$$
d(xv,x)=d(v,1)\leqslant K\cdot \partial (v,1)=K\cdot \partial (v,1)=K\cdot \partial (xv,x),
$$
while, for $a\notin V$,
\[\begin{split}
d(xa,x)
&\leqslant M\cdot \partial(xa,x)+N\\
&\leqslant M\cdot \partial(xa,x)+\frac{N\cdot \partial(a,1)}{\inf(\partial(y,1)\mid y\notin V)}\\
&\leqslant \Big(M+\frac{N}{\inf(\partial(y,1)\mid y\notin V)}\Big)\cdot \partial(xa,x).
\end{split}\]
So $d\leqslant L\cdot \partial$, where  $L=\max\Big\{K,M+\frac{N}{\inf(\partial(y,1)\mid y\notin V)}\Big\}$ and by symmetry we find that $d$ and $\partial$ are bi-Lipschitz equivalent.
\end{proof}

Thus, if $G$ admits a metric that is simultaneously minimal and maximal, then this defines a canonical global Lipschitz geometric structure on $G$.  To characterise this situation, we need a few new concepts from \cite{coarse}. 

A topological group is {\em Baire} if it satisfies the Baire category theorem, that is, if the intersection of countably many dense open sets is dense in $G$. Also, we say that  $G$ is {\em European} if it is Baire and, for every identity neighbourhood $V$, there is is a countable set $D\subseteq G$ so that $G=\langle V\cup D\rangle$. Clearly every Polish group and every connected completely metrisable group, e.g., the additive group $(X,+)$ of a Banach space, is European.  Also, a locally compact Hausdorff group is European if and only if it is $\sigma$-compact.

A subset $B$  of a topological group $G$ is {\em coarsely bounded} if it has finite diameter in every continuous left-invariant pseudometric on $G$. If $G$ is European, this is equivalent to asking that, for every identity neighbourhood $V$, there is a finite set $F\subseteq G$ and a $k$ so that $B\subseteq (FV)^k$.

Now, as opposed to minimal metrics, we do have a characterisation of the existence of maximal metrics. Namely, as shown in \cite{coarse}, a metrisable European group $G$ admits a maximal metric $d$ if and only if it is algebraically generated by a coarsely bounded set, which furthermore may be taken to be an identity neighbourhood $V$. Moreover, in this case, the maximal metric $d$ will be quasi-isometric to the word metric 
$$
\rho_V(g,f)=\min(k\mid \e v_1, \ldots, v_k\in V^\pm\colon  g=fv_1\cdots v_k)
$$
associated with $V$.

\begin{thm}\label{min-max}
Let $G$ be a Polish or, more generally, a European topological group  generated by a coarsely bounded set and assume that $G$ has a minimal metric.
Then $G$ has a metric that is simultaneously maximal and minimal and hence $G$ has a well-defined Lipschitz geometric structure.
\end{thm}

\begin{proof}Fix a minimal metric $d$ on $G$ and a coarsely bounded identity neighborhood $V$ generating $G$. Now, as shown in \cite{coarse}, the formula
$$
\partial(g,f)=\inf\big(\sum_{i=1}^nd(v_i, 1)\mid v_i\in V\;\&\; g=fv_1\cdots v_n\big)
$$
defines a compatible left-invariant metric on $G$, which is quasi-isometric to the word metric $\rho_V$. It thus follows that $\partial$ is quasi-isometric to a maximal metric on $G$ and therefore maximal itself. 

Observe now that, if $W$ is a symmetric identity neighbourhood so that $W^2\subseteq V$, then any two elements of $W$ differ on the right by an element of $V$ and so the two metric $d$ and $\partial$ agree on $W$. It therefore follows that $\partial$ is also minimal.
\end{proof}

Outside of the class of locally compact groups, the problem of determining which groups admit a minimal metric is unsolved. However, as is evident from Condition (\ref{strongest cond}) of Theorem \ref{main}, a group $G$ with a minimal metric must be {\em uniformly NSS} in the sense of the following definition.
\begin{defi} 
A topological group $G$ is  {\em uniformly NSS} group if there is an open set $U\ni 1$ so that, for every open $V\ni 1$, there is some $n$ for which
$$
g, g^2, \ldots, g^n\in U \;\;\saa\;\; g\in V.
$$
\end{defi}
As noted by P. Enflo \cite{enflo}, a uniformly NSS group is metrisable. Indeed, 
let $U\ni 1$ be as in the definition of the uniform NSS property and pick open neighbourhoods $W_n\ni1 $ so that $(W_n)^n\subseteq U$. Now, suppose $V\ni 1$ is open and let $n$ be such that $g\in V$ whenever $g, g^2, \ldots, g^n\in U$. Then clearly $g\in W_n$ implies that $g\in V$, i.e., $W_n\subseteq V$. Thus, the sets $W_n$ form a countable neighbourhood basis at $1$ and $G$ is metrisable by the result of Birkhoff and Kakutani. 

For other interesting facts about uniformly NSS groups, including that every Banach-Lie group is uniformly NSS, one may consult the paper \cite{pestov} by S. A. Morris and V. Pestov. In particular, the authors show that uniformly NSS groups are {\em locally minimal}, which we shall not define here. However, by essentially the same proof, we may prove the following.

\begin{prop}
Suppose $G\curvearrowright X$ is a continuous isometric action of a uniformly NSS topological group $G$, as witnessed by an identity neighbourhood $U$, on a metric space $(X,d)$. Assume also that, for some $\eps>0$ and $x\in X$, we have
$$
g\notin U\saa d(gx,x)>\eps.
$$
Then the orbit map 
$$
g\in G\mapsto gx\in X
$$
is a uniform embedding of $G$ into $X$.
\end{prop}

\begin{proof}As the orbit map is easily uniformly continuous, to see that it is a uniform embedding of $G$ into $X$, it suffices to  show that, for every identity neighbourhood $V$, there is $\eta>0$ so that 
$d(gx,x)\geqslant \eta$ whenever $g\notin V$. So let $V$ be given and pick $n$ so that $g\in V$ whenever $g, g^2, \ldots, g^n\in U$. Then, if $g\notin V$, there is $i\leqslant n$ so that $g^i\notin U$ and thus also $n\cdot d(gx,x)\geqslant d(g^ix,x)>\eps$. In other words, $d(g,1)\geqslant \frac \eps n$ for all $g\notin V$.
\end{proof}

Recall that a topological group $G$ is said to be a {\em SIN} group (for {\em small invariant neighbourhoods}) if there a a neighbourhood basis at the identity consisting of conjugacy invariant sets. In the context of metrisable groups, these are, by a result of V. Klee \cite{klee}, simply the groups admitting a compatible bi-invariant metric.

\begin{cor}\label{SIN+minimal}
Let $G$ be a SIN group with a minimal metric. Then $G$ admits a bi-invariant minimal metric.
\end{cor}

\begin{proof}
Suppose $d$ is a left-invariant minimal metric on $G$ as witnessed by an open set $U\ni1 $
so that
$$
g,g^2,\ldots, g^n\in U\;\saa\; d(g,1)\leqslant \frac 1n.
$$
Since $G$ is SIN, we may assume that $U$ is conjugacy invariant. Also, replacing $d$ with $\min\{d,1\}$, we can assume that $d\leqslant 1$. Define now a metric $\partial$ by
$$
\partial(g,h)=\sup_{f\in G}d(gf,hf),
$$
and note that, as $G$ is SIN, $\partial$ is a compatible bi-invariant metric on $G$. We claim that $\partial$ is minimal. Indeed, supposing that $g,g^2, \ldots, g^n\in U$, then, for every $f\in G$, we have $f\inv gf, (f\inv gf)^2, \ldots, (f\inv gf)^n\in U$ and thus $d(gf,f)=d(f\inv gf, 1)\leqslant 1/n$, i.e., $\partial(g,1)\leqslant 1/n$. 
\end{proof}

\begin{defi}
$G$ is a {\em locally SIN} group if there is an identity  neighbourhood $\ku O$ so that the sets
$$
V^\ku O=\{gfg\inv \del g\in \ku O\;\&\; f\in V\},
$$
where $V$ varies over identity neighbourhoods, form a neighbourhood basis at the identity.
\end{defi}
We claim that $G$ is locally SIN if and only if the inversion map $g\mapsto g\inv$ is left-uniformly continuous on an open symmetric set $W\ni1 $. Indeed, suppose first that  inversion is left-uniformly continuous on $W$. This means that, for  all open $V\ni 1$ there is an open $U\ni1 $ so that 
$$
g,f\in W\;\;\&\;\; g\inv f\in U\;\;\saa\;\; gf\inv \in V.
$$
We let $\ku O\ni1 $ be symmetric open so that $\ku O^2\subseteq W$. Then, for every open $V\ni 1$, pick $U\subseteq \ku O$ as above.  Then, if $g\in \ku O\subseteq W$ and $h\in U$, note that also $f=gh\in \ku OU\subseteq W$ and $g\inv f=h\in U$, whereby $ghg\inv =gf\inv\in V$, showing that $U^\ku O\subseteq V$, whence $G$ is locally SIN.
 
Conversely, suppose that $G$ is locally SIN as witnessed by some symmetric open $\ku O\ni 1$. Then, if $V\ni1$ is symmetric open, find some open $U$ with $1\in U\subseteq \ku O$ and $U^\ku O\subseteq V$ and note that 
$$
g,f\in \ku O\;\;\&\;\; g\inv f\in U\;\saa\; fg\inv =g\cdot g\inv f\cdot g\inv \in V\;\saa\; gf\inv =(fg\inv)\inv \in V.
$$
So inversion is left-uniformly continuous on $\ku O$.

Similarly, one may show that $G$ is locally SIN if and only if there is an open set $W\ni 1$ so that the map $(g,f)\in W\times W\mapsto gf\in W^2$ is left-uniformly continuous.

For the next Proposition, we recall that a topological group $G$ has {\em property (OB)} if, whenever $G\curvearrowright (X,d)$ is a continuous isometric action on a metric space $(X,d)$, every orbit is bounded. If $G$ is separable metrisable, property (OB) is equivalent to the following property (see \cite{OB}), which we may term the {\em strong property (OB)}: For every open $\ku O\ni1$ there are a finite set $F\subseteq G$ and a $k$ so that $G=(F\ku O)^k$. 

Apart from compact groups, a surprisingly large number of topological groups have property (OB). Of particular interest to us is the unitary group $U(\ku H)$ of separable infinite-dimensional Hilbert space, which has the strong property (OB) when equipped with the operator norm topology. This follows from the spectral theorem. More examples can be found, e.g., in \cite{OB}.

\begin{prop}\label{SIN+OB}
Let $G$  be a metrisable locally SIN group with the strong property (OB). Then $G$ is SIN. \end{prop}

\begin{proof}
Let $\ku O$ be an open neighbourhood of $1$ witnessing that $G$ is locally SIN and pick $k$ and a finite subset $F\subseteq G$ so that $G=(F\ku O)^k$. Suppose now that $U\ni1 $ is open. Set $V=\bigcap_{f\in F}f\inv Uf$, which is an open neighbourhood of $1$, and let $W\ni 1$ be an open set so that $W^\ku O\subseteq V$, whence $W^{F\ku O}=(W^\ku O)^F\subseteq V^F\subseteq U$. Thus, by induction, we can choose open $U=W_0\supseteq W_1\supseteq W_2\supseteq  \ldots\supseteq W_k\ni1 $ so that $W_{i}^{F\ku O}\subseteq W_{i-1}$,
whereby $W_k^G=W_k^{(F\ku O)^k}\subseteq W_0=U$. In particular, $W_k^G$ is a conjugacy invariant neighbourhood of $1$ contained in $U$, which shows that $G$ admits a neighbourhood basis at $1$ consisting of conjugacy invariant sets.
\end{proof}

Enflo \cite{enflo} showed that uniformly NSS groups are locally SIN (though he used the terminology {\em locally uniform} in place of {\em locally SIN}). To see this, let $U\ni 1$ be the open set given by the uniform NSS property and pick a symmetric open $\ku O\ni 1$ so that $\ku O^3\subseteq U$. Suppose now $W$ is an arbitrary neighbourhood of $1$ and find $n$ so that
$$
g, g^2, \ldots, g^n\in U\;\; \saa\;\;g\in W.
$$ 
We now choose some open $V\ni 1$ so that $V^n\subseteq \ku O$, whence also $\ku OV^n\ku O\inv \subseteq U$. In particular,
if $v\in V$ and $g\in \ku O$, then $(gvg\inv)^m=gv^mg\inv\in U$ for $m=1,\ldots, n$, 
whereby $gvg\inv \in W$. In other words, $V^\ku O\subseteq W$, verifying that $G$ is locally SIN.

So combining this result of Enflo with Corollary \ref{SIN+minimal} and Proposition \ref{SIN+OB}, we obtain the following.

\begin{cor}
Every group with a minimal metric and the strong property (OB) also has a bi-invariant minimal metric.
\end{cor}

Recall that a sequence $(f_i)$ in a metrisable group $G$ is  {\em left-Cauchy} if $f_i\inv f_j\Lim{i,j\til \infty}1$ and {\em right-Cauchy} if $f_i f_j\inv \Lim{i,j\til \infty}1$. The group $G$ is {\em Ra\u\i kov complete} if every sequence that is both left and right-Cauchy is convergent. This is equivalent to $G$ being {\em completely metrisable}, i.e., that the topology on $G$ can be induced by a complete metric. 
Also, $G$ is {\em Weil complete} if every left-Cauchy sequence in $G$ is convergent. This, in turn, is equivalent to the existence of a compatible complete left-invariant metric on $G$ (such groups are sometime denoted {\em CLI} for {\em complete left-invariant}). Moreover, in this case, every compatible left-invariant metric is complete.

\begin{lemme}\label{weil}
If $G$ is locally SIN and completely metrisable, then $G$ is Weil complete. In particular, every minimal metric on a completely metrisable group is necessarily complete.
\end{lemme}

\begin{proof}
Indeed, suppose that $\ku O\ni 1$ is an open set witnessing that $G$ is locally SIN and that $(f_i)$ is left-Cauchy. To see that $(f_i)$ is right-Cauchy, fix a neighbourhood $U$ of $1$ and pick some $i_0$ so that $f_i\inv f_j\in \ku O$ for all $i,j\geqslant i_0$. Now choose some open $W\ni1 $ so that $f_{i_0}Wf_{i_0}\inv \subseteq U$ and some open $V\ni 1$ so that $V^\ku O\subseteq W$. Finally, let $i_1\geqslant i_0$ be so that $f_i\inv f_j\in V$ whenever $i,j\geqslant i_1$. Then $i,j\geqslant i_1$ implies that $f_{i_0}\inv f_i\in\ku O$ and so $f_i=f_{i_0}g$ for some $g\in \ku O$, whence
$$
f_jf_i\inv= f_i\cdot f_i\inv f_j\cdot f_i\inv =  f_{i_0}g \cdot f_i\inv f_j\cdot  g\inv  f_{i_0}\inv \in  f_{i_0}g Vg\inv  f_{i_0}\inv \subseteq f_{i_0}W  f_{i_0}\inv\subseteq U.
$$
Thus $(f_i)$ is also right-Cauchy and therefore convergent in $G$.
\end{proof}

The next result has a long history and many variations. The first occurrence seems to be the paper by A. Gleason \cite{gleason2} in which it is proved that, in a locally euclidean NSS group, there is an identity neighbourhood in which square roots, whenever they exist, are necessarily unique. We shall need a stronger version of this, namely that in a uniformly NSS group the extraction of square roots, whenever they exist, is left-uniformly continuous. A result of this form, under additional hypotheses, is also proved in Enflo's paper \cite{enflo}.
 
\begin{lemme}\label{square roots}
Suppose $G$ is uniformly NSS. Then there is an open set $V\ni1$ so that, for every open $U\ni1$, there is an open $W\ni 1$ so that
$$
g,f\in V\;\;\;\&\;\;\;g^{-2}f^2\in W\;\saa\; g\inv f\in U.
$$
In particular, the map $g\mapsto g^2$ is injective on $V$.
 \end{lemme}

\begin{proof}Since uniformly NSS groups are also locally SIN, we fix a symmetric open set $\ku O\ni1 $ witnessing both that $G$ is uniformly NSS and locally SIN. Let also $V\ni 1$ be symmetric open so that $\big(V^\ku OVV\big)^2\subseteq \ku O$. 

To see that the lemma holds for $V$, suppose $U$ is given and pick some $n$ so that $\big(y, y^2, \ldots, y^n\in \ku O\;\saa\; y\in U\big)$. Let also  $W\ni 1$ be an open set with $WW^\ku O W^{\ku O^2}\cdots W^{\ku O^{n-1}}\subseteq V$. Then, for $x,y\in \ku O$, 
$$
x^{-i}y^{i}=\big[x\inv y\big]\cdot \big[y\inv(x\inv y)y\big]\cdot \big[y^{-2}(x\inv y)y^2\big]\cdots 
\big[y^{-(i-1)}(x\inv y)y^{i-1}\big]
$$
and so, if also $x\inv y\in W$ and $i\leqslant n$, then $x^{-i}y^i\in WW^\ku O W^{\ku O^2}\cdots W^{\ku O^{i-1}}\subseteq V$.
In other words, 
$$
x,y\in \ku O\;\;\;\&\;\;\; x\inv y\in W\;\;\saa\;\; \a i\leqslant n\colon\; x^{-i}y^i\in V.
$$

Suppose that $g,f\in V$ satisfy $g^{-2}f^2\in W$ and set $y=g\inv f$ and $x=g\inv y\inv g$. Then $x\inv y=g^{-2}f^2\in W$ and thus $g\inv y^igy^i =x^{-i}y^i\in V$ for $i\leqslant n$. Note now that
$$
y^i\in \ku O\;\;\saa\;\; y^{2i}=(y^ig\inv y^{-i})\cdot g\cdot (g\inv y^igy^i)\in V^\ku OVV
$$
for all $i\leqslant n$. We claim that $y^i\in V^\ku OVV$ for even $i\leqslant n$ and $y^i\in (V^\ku OVV)^2\subseteq \ku O$ for odd $i\leqslant n$. This is clear for $i=0,1$, so suppose the result holds for all $i\leqslant j<n$ and consider $i=j+1$. If $i$ is odd, then $j$ is even and thus $y^j\in V^\ku OVV$ and $y^i=yy^j\in V^2\cdot  V^\ku OVV\subseteq (V^\ku OVV)^2$. On the other hand, if $i$ is even, then $l=\frac i2\leqslant j$ and so $y^l\in (V^\ku OVV)^2\subseteq \ku O$, whence $y^i=y^{2l}\in V^\ku OVV$. 

Thus, $y,y^2,\ldots, y^n\in \ku O$, whereby $g\inv f=y\in U$ as claimed.
\end{proof}

Our final result originates in work of C. Chevalley \cite{chevalley}, who showed how to construct one-parameter subgroups in locally euclidean NSS groups. Again, a generalisation to the non-locally compact setting was obtained by Enflo in \cite{enflo}, in which he proved the result below under the assumption that the group is {\em uniformly dissipative}. However, this assumption excludes, for example, compact Lie groups and therefore does not generalise the classical setting of locally compact Lie groups. Our result below includes this latter setting.

\begin{thm}\label{one-parameter}
Suppose $G$ is a completely metrisable topological group admitting a minimal metric and that, for every open $W\ni1$, the set $\{g^2\del g\in W\}$ is dense in a neighbourhood of $1$. Then there are open sets $\ku U\supseteq \ku O\ni1$ so that, for every $f\in \ku O$, there is a unique one-parameter subgroup $(h^\alpha)_{\alpha\in \R}$ with $h^1=f$ and $h^\alpha\in \ku U$ for all $\alpha\in [-1,1]$.
\end{thm}

\begin{proof}
Let $d$ be a minimal metric and observe that $d$ is complete by Lemma \ref{weil}.
We claim that, for all sufficiently small identity neighbourhoods $V$, there is an identity neighbourhood $W$ so that every element of $W$ has a square root in $V$. 

Note first that it suffices to show this for closed $V$. Since $G$ has a minimal metric, it is uniformly NSS and hence, by Lemma \ref{square roots} it follows that, if $V$ is a sufficiently small neighbourhood of $1$, then, for every open $U\ni 1$, there is an open $\tilde U\ni1$ satisfying
\begin{equation}\label{uniformity}
g,f\in V\;\;\&\;\; g^{-2}f^2\in\tilde U\;\saa\; g\inv f\in U.
\end{equation}

Given such a closed neighbourhood $V$, set $W=\ov {\{g^2\del g\in V\}}$, which by assumption is a neighbourhood of $1$, and assume that $f\in W$. We pick $g_n\in V$ so that $g_n^2\Lim{n} f$, which means that $(g_n^2)$ is left-Cauchy and thus, by the above assumption (\ref{uniformity}) on $V$,  also $(g_n)$ is left-Cauchy. Since $d$ is a complete left-invariant metric and $V$ is closed, it follows that $(g_n)$ is convergent to some $g\in V$, whence $g^2=\lim_{n}g_n^2=f$. In other words, every element of $W$ has a square root in $V$, which proves our claim.

Since $d$ is minimal, there are an identity  neighbourhood $V_0$ and a $k\geqslant 1$ so that, for all $g\in G$ and $n\geqslant 1$, 
$$
g, g^2, g^4, g^8, \ldots, g^{2^n}\in V_0\;\saa\; d(g,1)\leqslant 2^{k-n-1}\cdot d(g^{2^n},1).
$$
Define inductively identity neighbourhoods $V_0\supseteq V_1\supseteq \ldots\supseteq V_k$  so that every element of $V_{i+1}$ is the square of some element in $V_i$. By shrinking $V_k$, we may assume that $V_k=B_d(\eps)$ for some $\eps>0$.

Now suppose that $h_0\in V_k$ and choose inductively $h_i\in V_{k-i}$ so that $h_{i+1}^2=h_i$. Then $h_k, h_k^2=h_{k-1}, h_k^4=h_{k-2}, \ldots, h_k^{2^k}=h_0$ all belong to $V_0$, whence 
$$
d(h_k,1)\leqslant 2^{k-k-1}\cdot d(h_k^{2^k}, 1)=1/2 \cdot d(h_0, 1)\leqslant \eps/2.
$$
This shows that every $h\in V_k$ has a $2^k$-th root $f \in V_k$ so that $d(f, 1)\leqslant 1/2 \cdot d(h,1)$. 

Therefore, if $f\in V_k$ is given, we can choose an infinite sequence $h_0,h_1, h_2,\ldots\in V_k$  beginning at $h_0=f$ so that $h_{i+1}^{2^k}=h_i$ and $d(h_{i+1},1)\leqslant 1/2\cdot d(h_i,1)$ for all $i$. In particular, $h_i^{2^{k\cdot j}}=h_{i-j}$ for all $j\leqslant i$.  For every dyadic rational number $\alpha=\frac m{2^{k\cdot i}}$ with $m\in \Z$ and $i\in \N$, we can then unambigously define $h^\alpha=h_i^m$ and see that $h^\alpha\cdot h^\beta=h^{\alpha+\beta}$ for all dyadic rationals $\alpha$ and $\beta$. 

Now, if $\alpha\in [0,1[$ is a dyadic rational, write
$$
\alpha=\frac {a_1}{2^k}+\frac{a_2}{2^{k\cdot 2}}+\ldots+\frac{a_p}{2^{k\cdot p}}
$$
with  $0\leqslant a_i<2^k$. Then
$$
h^\alpha= h_1^{a_1}\cdot h_2^{a_2}\cdots h_p^{a_p},
$$
whereby
\[\begin{split}
d(h^\alpha,1)
&\leqslant a_1\cdot d(h_1,1)+a_2\cdot d(h_2,1)+\ldots+a_{p}\cdot d(h_{p},1)\\
&\leqslant \Big(\frac{a_1}2+\frac{a_2}{2^2}+\ldots+\frac{a_{p}}{2^{p}}\Big)\cdot \eps\\
&<2^k\cdot \eps.
\end{split}\]
Moreover, if $\alpha<\frac1{2^{ki}}$, then $a_1=\ldots=a_i=0$, whence $d(h^\alpha,1)<2^{k-i}\cdot \eps$.

It follows that the mapping $\alpha\mapsto h^\alpha$ is a continuous homomorphism from the additive group of dyadic rationals with the topology induced by $\R$ into $G$. Since $d$ is a complete metric on $G$, it follows that this extends to a continuous one-parameter subgroup $(h^\alpha)_{\alpha\in \R}$ with $h^1=f$ and so that $d(h^\alpha,1)\leqslant 2^k\cdot \eps$ for all $\alpha\in [-1,1]$.

Now, suppose that $(h^\alpha)_{\alpha\in \R}$ and $(g^\alpha)_{\alpha\in \R}$ are distinct one-parameter subgroups in $G$ with $h^1=g^1$. Then, by the density of the dyadic rationals in $\R$, there must be some dyadic rational $\alpha=\frac 1{2^{n}}$ so that $h^\alpha\neq g^\alpha$. However, as $h^1=g^1$, it follows that there is an $\ell\geqslant 0$ so that $h^{\frac1{2^{\ell+1}}}\neq g^{\frac1{2^{\ell+1}}}$, while $h^{\frac1{2^{\ell}}}=g^{\frac1{2^{\ell}}}$. So the squaring map $f\mapsto f^2$ fails to be injective on any set containing $\{h^\beta\}_{\beta\in [-1,1]}$ and $\{g^\beta\}_{\beta\in [-1,1]}$.
 
Therefore, if we choose $V_k=B_d(\eps)$ small enough so that the squaring map is injective on $B_d(2^k\cdot \eps)$, which is possible by Lemma \ref{square roots}, then, for every $f\in V_k$, there is a unique one-parameter subgroup $(h^\alpha)_{\alpha\in \R}\subseteq G$ so that $h^1=f$ and $h^\alpha\in B_d(2^k\cdot \eps)$ for all $\alpha\in [-1,1]$. Setting $\ku U= B_d(2^k\cdot \eps)$ and $\ku O=B_d(\eps)$, the theorem follows.
\end{proof}



\end{document}